\documentclass[twoside,a4paper,11pt]{amsart}

\usepackage[all]{xy}
\usepackage{amssymb,latexsym,enumitem,tikz}
\usepackage[backref,colorlinks]{hyperref}

\newtheorem{thm}{Theorem}[section]

\newtheorem{cor}[thm]{Corollary}
\newtheorem{prop}[thm]{Proposition}

\DeclareMathOperator{\Ext}{Ext}

\author{Luke Kershaw} 
\address{School of Mathematics, University of Bristol, Bristol BS8 1UG, UK}
\email{l.kershaw@bristol.ac.uk}
\author{Jeremy Rickard} 
\address{School of Mathematics, University of Bristol, Bristol BS8 1UG, UK}
\email{j.rickard@bristol.ac.uk}
\date\today
\title[]{A finite dimensional algebra with infinite delooping level}

\begin{document}
\maketitle

\begin{abstract}
  We give an example of a finite dimensional algebra with infinite delooping
  level, based on an example of a semi-Gorenstein-projective module due to
  Ringel and Zhang.
\end{abstract}

\section{Introduction}
\label{sec:introduction}

One of the most celebrated open problems in the representation theory of finite
dimensional algebras is the finitistic dimension conjecture, publicized by
Bass~\cite{MR157984} in 1960, and subsequently proven to imply a host of other
homological conjectures.

In recent years other, potentially stronger, conditions that would imply the
finitistic dimension conjecture have been considered. One is the notion of
``injective generation''~\cite{MR3982972}, that we consider briefly in
Section~\ref{sec:other-prop-lambd}. But our main focus is the finiteness of
delooping level, introduced by G\'{e}linas~\cite{MR4355734}. He works
in greater generality, but since the main example in this paper is a finite
dimensional algebra, we shall describe the delooping level in that context.

Given a finite dimensional algebra $A$ and a finitely generated right $A$-module
$M$, the delooping level of $M$ is the smallest nonnegative integer $n$ such
that $\Omega^{n}M$ is a direct summand (in the stable module category) of
$\Omega^{n+1}N$ for some other finitely generated module $N$, or is infinite if
no such $n$ exists. The delooping level is denoted by $\operatorname{dell}M$.

Then the delooping level $\operatorname{dell}A$ is defined to be the largest
delooping level of a simple $A$-module.

One of the main properties of the delooping level is that 
$\operatorname{dell}A$ is an upper bound for the (big) finitistic dimension
of the opposite algebra $A^{op}$ and so if
$\operatorname{dell}A<\infty$ then the big finitistic dimension conjecture holds
for $A^{op}$~\cite[Proposition 1.3]{MR4355734}.

In this paper we shall give what we believe to be the first known example of a
finite dimensional algebra with infinite delooping level. This does not give a
counterexample to the finitistic dimension conjecture.

Although the delooping level of an algebra $A$ is defined in terms of the
delooping levels of \emph{simple} modules, we show in
Section~\ref{sec:modul-with-infin} that if there is any finitely generated
$A$-module $M$ with $\operatorname{dell}M=\infty$, then there is another finite
dimensional algebra $B$ with $\operatorname{dell}B=\infty$.

Our main result, Proposition~\ref{prop:delooping-level-mq-3}, is that a
particular ``semi-Gorenstein-projective'' module studied by Ringel and
Zhang~\cite{MR4076806} has infinite delooping level. Gorenstein projective
modules are easily seen to have delooping level equal to zero, but the same
argument does not apply to semi-Gorenstein-projective modules. What inspired us
to look at this example is that G\'{e}linas showed in~\cite[Theorem
1.10]{MR4355734} that the module $N$ in the definition of delooping level can
always be taken to be $\Sigma^{n+1}\Omega^{n}M$ (where $\Sigma$ denotes the left
adjoint of $\Omega$), which in the case of a semi-Gorenstein-projective module
simplifies to $\Sigma M$, so that proving $\operatorname{dell}M=\infty$ reduces
to proving that $\Omega^{n}M$ is never a summand of $\Omega^{n}(\Omega\Sigma)M$.

There are other examples known of modules that are semi-Gorenstein-projective,
but not Gorenstein projective. The first was given by Jorgensen and \c{S}ega,
and Marczinzik, who calls such modules ``stable'', gave
others~\cite{MR2238367,marczinzik:2017}. It might be interesting to study
the delooping level of these examples.

\vspace{2mm}

\noindent \textbf{Acknowledgments.}  The first author was supported by
Engineering and Physical Sciences Research Council Doctoral Training Partnership
award EP/R513179/1.

\section{Delooping levels of one point extensions}
\label{sec:modul-with-infin}

\begin{prop}
  If $M$ is a finitely generated module with infinite delooping level
  for a finite dimensional $k$-algebra $A$, then there is a finite dimensional
  $k$-algebra $B$ with infinite delooping level.
\end{prop}

\begin{proof}
  Let $B$ be the one point extension algebra
  $$B=A[M]=
  \begin{pmatrix}
    k&M\\0&A
  \end{pmatrix},
  $$
  and let $S$ be the simple $B$-module $
  \begin{pmatrix}
    k&M
  \end{pmatrix}/
  \begin{pmatrix}
    0&M
  \end{pmatrix}
  $

  Then for $n\geq1$, $\Omega^{n}S=
  \begin{pmatrix}
    0&\Omega^{n-1}M
  \end{pmatrix}
  $. If $S$ has finite delooping level, then for all sufficiently large $n$,
  $\Omega^{n}S$ is a summand of $\Omega^{n+1}X$ for some $B$-module
  $X$, and $\Omega X=
\begin{pmatrix}
  0&N
\end{pmatrix}
$ for some $A$-module $N$, and $\Omega^{n+1}X=
\begin{pmatrix}
  0&\Omega^{n}N
\end{pmatrix}
$.

Therefore $\Omega^{n-1}M$ is a summand of $\Omega^{n}N$, and so $M$
has finite delooping level.
\end{proof}

\section{Ringel and Zhang's example}
\label{sec:ring-zhangs-example}

In~\cite{MR4076806}, Ringel and Zhang exhibited an example of a
three-dimensional module $M$ for a six-dimensional local finite dimensional
algebra $\Lambda$ that is semi-Gorenstein-projective (meaning that
$\Ext^{i}(M,\Lambda)=0$ for all $i\geq1$), but is not torsionless (i.e., is not
a submodule of a projective $\Lambda$-module), and hence not
Gorenstein-projective.

In this section, we gather the information that we will need about this
example. All of the results in this section are taken from Ringel and Zhang's
paper, except for the remarks about the subalgebra $\Gamma$. Note that they work
with left modules, whereas we are using right modules, so the algebra we define
is the opposite of theirs.

\subsection{The algebra $\Lambda=\Lambda(q)$ and the subalgebra
  $\Gamma$}~\cite[\S6.1]{MR4076806}
\label{sec:algebra-lambd}

Let $k$ be a field, and $q\in k$ an element with infinite multiplicative
order. Then $\Lambda=\Lambda(q)$ is the algebra
$$k\langle x,y,z\rangle/(x^{2},y^{2},z^{2},zy,yx+qxy,zx-xz,yz-xz),$$
which is easily seen to be six-dimensional, with basis $1,x,y,z,yx,zx$.

Let $\Gamma$ be the two-dimensional subalgebra of $\Lambda$ generated by
$x$. Note that $\Lambda$ is a free right $\Gamma$-module.

\subsection{The modules $M(\alpha)$}~\cite[\S6.1,\S6.3]{MR4076806}
\label{sec:modules-malpha}

For each $\alpha\in k$, $M(\alpha)$ is the three-dimensional $\Lambda$-module
with basis $v,v',v''$, where $vx=\alpha v'$, $vy=v'$, $vz=v''$, with $v'$ and
$v''$ annihilated by $x$, $y$ and $z$.

Next, we gather the results that we need regarding the modules $M(\alpha)$. We
shall state them in terms of the effect of $\Omega$ and $\Sigma$ on the modules,
but apart from the brief remark about the subalgebra $\Gamma$, everything in the
following proposition is a restatement of results of Ringel and
Zhang~\cite{MR4076806}.

\begin{prop}\label{sec:modules-malpha-1}
  Let $\alpha\in k$.
  \begin{enumerate}
  \item If $\alpha\neq1$ then $\Omega M(\alpha)=M(q\alpha)$.
  \item If $\alpha\neq q$ then $\Sigma M(\alpha)=M(q^{-1}\alpha)$.
  \item $\Omega\Sigma M(q)$ is a two-dimensional module that is free as a $\Gamma$-module. 
  \end{enumerate}
\end{prop}

\begin{proof}
  Both (1) and (2) follow immediately from~\cite[Lemma 6.4 and Lemma
  3.2]{MR4076806}. Note that Ringel and Zhang use the notation $\mho
  M(\alpha)$ where we use $\Sigma M(\alpha)$.

  For any module $M$, $\Omega\Sigma M$ is the maximal torsionless quotient of
  $M$. For $M=M(q)$, this is shown in~\cite[Lemma 6.2]{MR4076806} to be
  $M(q)/M(q)z$, which is freely generated as a $\Gamma$-module by $v$.
\end{proof}

\section{The delooping level of $M(q)$}
\label{sec:delooping-level-mq}

\begin{prop}\label{prop:delooping-level-mq-3}
  The $\Lambda$-module $M(q)$ has infinite delooping level.
\end{prop}

\begin{proof}
  G\'{e}linas showed that if a module $M$ has delooping level $n<\infty$, then
  $\Omega^{n}M$ is a stable direct summand of
  $\Omega^{n+1}\Sigma^{n+1}\Omega^{n}M$~\cite[Theorem 1.10]{MR4355734}. We shall
  show that this is not the case for $M=M(q)$.

  By Proposition~\ref{sec:modules-malpha-1}(1), $\Omega^{n}M(q)\cong
  M(q^{n+1})$ is three-dimensional.

  By Proposition~\ref{sec:modules-malpha-1}(2),
  $\Sigma^{n}\Omega^{n}M(q)\cong M(q)$, so
  $$\Omega^{n+1}\Sigma^{n+1}\Omega^{n}M\cong \Omega^{n+1}\Sigma M(q).$$

  But since $\Omega\Sigma M(q)$ and $\Lambda$
  are both free as $\Gamma$-modules,
  $$\Omega^{n+1}\Sigma M(q)=\Omega^{n}\left(\Omega\Sigma M(q)\right)$$
  is also free as a $\Gamma$-module and therefore every direct summand has even
  dimension.
\end{proof}

\begin{cor}
  The one-point extension algebra $\Lambda[M(q)]$ has infinite delooping level.
\end{cor}

\begin{proof}
  Proposition~\ref{sec:modules-malpha-1} and
  Proposition~\ref{prop:delooping-level-mq-3}.
\end{proof}

\section{Other properties of $\Lambda[M(q)]$}
\label{sec:other-prop-lambd}

In this final section, we consider some other related properties of the algebra
$\Lambda[M(q)]$.

One of the main reasons that the delooping level was introduced was its
connection with the finitistic dimension conjecture: for Artinian rings $A$,
G\'{e}linas proved that the delooping level $\operatorname{dell}A$ of $A$ is an
upper bound for the big finitistic dimension $\operatorname{Findim}A^{op}$ of
its opposite algebra, so that finiteness of the delooping level of $A$ implies
the big finitistic dimension conjecture for $A^{op}$~\cite[Proposition
1.3]{MR4355734}.

It is therefore natural to wonder whether $\Lambda[M(q)]^{op}$ might have infinite
finitistic dimension. This is not the case.

\begin{prop}
  $\operatorname{Findim}\Lambda[M(q)]^{op}=1$
\end{prop}

\begin{proof}
  Since
  $$\Lambda[M(q)]^{op}=
  \begin{pmatrix}
    k&0\\M(q)&\Lambda^{op}
  \end{pmatrix}
  $$
  there is an obvious module with projective dimension equal to one -- namely,
  $
  \begin{pmatrix}
    0&\Lambda^{op}
  \end{pmatrix}=
  \begin{pmatrix}
    M(q)&\Lambda^{op}
  \end{pmatrix}/
  \begin{pmatrix}
    M(q)&0
  \end{pmatrix},
  $
  and so $1\leq\operatorname{Findim}\Lambda[M(q)]^{op}$.
  
  But since $\Lambda[M(q)]^{op}$ is a triangular matrix ring, it is
  well-known~\cite[Corollary 4.21]{MR0389981} that
  $$\operatorname{Findim}\Lambda[M(q)]^{op}\leq
  1+\operatorname{Findim}k+\operatorname{Findim}\Lambda^{op},$$ and since
  $\Lambda^{op}$ is a local finite dimensional algebra,
  $\operatorname{Findim}\Lambda^{op}=0$. So
  $\operatorname{Findim}\Lambda[M(q)]^{op}\leq1$.
\end{proof}

Another property that has recently been shown to imply the big finitistic
conjecture for a finite dimensional algebra $A$ is ``injective generation''. If
the unbounded derived category $\mathcal{D}(A)$ is equal to
$\operatorname{Loc}(\operatorname{Inj}A)$, the localizing subcategory of
$\mathcal{D}(A)$ generated by injective modules, then the big finitistic
dimension conjecture holds for $A$~\cite[Theorem 4.3]{MR3982972}.

We shall end by showing that the algebra $\Lambda[M(q)]^{op}$ is not a
counterexample to injective generation.

\begin{prop}
  Injectives generate for $\Lambda[M(q)]^{op}$.
\end{prop}

\begin{proof}
  Since $\Lambda[M(q)]^{op}$ is a triangular matrix ring, by a result of
  Cummings it is sufficient to show that injectives generate for
  $\Lambda^{op}$~\cite[Example 6.11]{MR4546142}.

  Since $M(0)=\Omega M(0)$, $M(0)$ has finite syzygy type as a $\Lambda$-module,
  and so, taking the vector space dual, $DM(0)$ has finite cosyzygy type as a
  $\Lambda^{op}$-module, which implies that $DM(0)$ is in
  $\operatorname{Loc}(\operatorname{Inj}\Lambda^{op})$~\cite[Proposition
  7.2]{MR3982972}.

 Since $DM(0)$ has radical length two, there is a short exact sequence
 $$0\to S_{1}\to DM(0)\to S_{2}\to 0$$
 where $S_{1}$ and $S_{2}$ are nonzero semisimple $\Lambda^{op}$-modules. Taking the
 coproduct of infinitely many copies of this sequence, we get a short exact
 sequence
 $$0\to S\to \bigoplus DM(0)\to S\to0$$
 where $S$ is an infinitely generated semisimple module. Splicing together
 copies of this sequence gives a resolution
 $$\cdots\to\bigoplus DM(0)\to\bigoplus DM(0)\to S\to 0$$
 by coproducts of copies of $DM(0)$, and so $S$ is in
 $\operatorname{Loc}(\operatorname{Inj}\Lambda^{op})$~\cite[Proposition
 2.1(h)]{MR3982972}. Therefore the unique simple $\Lambda^{op}$-module is in
 $\operatorname{Loc}(\operatorname{Inj}\Lambda^{op})$ and so
 $\mathcal{D}(\Lambda^{op})=\operatorname{Loc}(\operatorname{Inj}\Lambda^{op})$~\cite[Proposition
 2.1(e) and Lemma 6.1]{MR3982972}.
\end{proof}

\bibliography{deloop}{}
\bibliographystyle{amsalpha}

\end{document}